\documentclass[12pt]{article}
\usepackage{amsfonts}
\newtheorem{lemma}{Lemma}
\newtheorem{theorem}{Theorem}
\newtheorem{proposition}{Proposition}

\newenvironment{proof}{{\bf Proof}.}{\hfill $\Box$}

\begin{document}

%.dhahri0421.tex\bigskip\bigskip

\title{\sc On the quadratic Fock functor\footnote{Research partially supported by PBCT-ADI 13 grant ``Laboratorio de An\'alisis y Estoc\'astico'', Chile}}

\author{ Ameur Dhahri\\\vspace{-2mm}\scriptsize Centro de An\'alisis Estoc\'astico y Aplicaciones\\\vspace{-2mm}\scriptsize Pontificia Universidad Cat\'olica de Chile\\\vspace{-2mm}\scriptsize adhahri@puc.cl}
\date{}
\maketitle
\begin{abstract}
We prove that the quadratic second quantization of an operator $p$ on $L^2(\mathbb{R}^d)\cap L^\infty (\mathbb{R}^d)$ is an orthogonal projection on the quadratic Fock space if and only if $p$ is a multiplication operator by a characteristic function $\chi_I$, $I\subset\mathbb{R}^d$.
\end{abstract}
\section{Introduction}
The renormalized square of white noise (RSWN) was first introduced by Accardi-Lu--Volovich in \cite{AcLuVo}. Later, Sniady introduced the free RSWN white noise (cf \cite{Sn}). Subsequently, its relation with the L\'evy processes on real Lie algebras was established in \cite{AcFrSk}.

Recently, in \cite{AcDh1}--\cite{AcDh2}, the authors constructed the quadratic Fock functor. In particular, they characterized the operators on the one-particle Hilbert algebra whose quadratic second quantization is isometric (resp. unitary). A sufficient condition for the contractivity of the quadratic second quantization was derived too.

It is well known that the first order second quantization $\Gamma_1(p)$ of an operator $p$, defined on the usual Fock space, is an orthogonal projection if and only if $p$ is an orthogonal projection (cf \cite{Par}). In the present paper, it is shown that the set of orthogonal projections $p$, whose quadratic second quantization $\Gamma_2(p)$ is an orthogonal projection, is quite reduced. More precisely, we prove that $\Gamma_2(p)$ is an orthogonal projection if and only if $p$ is a multiplication operator by a characteristic function $\chi_I$, $I\subset\mathbb{R}^d$.  

This paper is organized as follows. In section \ref{section1}, we recall some basic properties of the quadratic Fock functor. The main result is proved in section \ref{section2}.

\section{Quadratic Fock functor}\label{section1}

The algebra of the renormalized square of white noise (RSWN) with test function Hilbert algebra
$$
{\cal A} := L^2(\mathbb{R^d})\cap L^\infty(\mathbb{ R^d})
$$
is the $*$-Lie-algebra, with central element denoted $1$, generators $B^+_f,B_h, N_g ,\, \, f,g,h\in L^2(\mathbb{R}^d)\cap L^\infty(\mathbb{R}^d)\}$, involution 
$$
(B^+_f)^*=B_f \qquad , \qquad N_f^*=N_{\bar{f}}
$$ 
and commutation relations
\begin{eqnarray}\label{commutation}
[B_f,B^+_g]=2c\langle f,g\rangle+4N_{\bar fg},\,\;[N_a,B^+_f]=2B^+_{af}
\end{eqnarray}
$$[B^+_f,B^+_g]=[B_f,B_g]=[N_a,N_{a'}]=0,$$
for all $a$, $a'$, $f$, $g\in L^2(\mathbb{R}^d)\cap L^\infty(\mathbb{R}^d)$.

The Fock representation of the RSWN is characterized by a cyclic vector $\Phi$ satisfying 
$$
B_f\Phi=N_g\Phi=0
$$ 
for all $f,g\in L^2(\mathbb{R}^d)\cap L^\infty(\mathbb{R}^d)$ (cf \cite{AcAmFr}, \cite{AcFrSk}). 

\subsection{Quadratic Fock space}

In this subsection, we recall some basic definitions and properties of the quadratic exponential vectors and the quadratic Fock space. We refer the interested reader to \cite{AcDh1}--\cite{AcDh2} for more details.

The quadratic Fock space $\Gamma_2(L^2(\mathbb{R}^d)\cap L^\infty(\mathbb{R}^d))$ is the closed linear span of $\big\{B^{+n}_f\Phi$, $n\in\mathbb{N}$, $f\in L^2(\mathbb{R}^d)\cap L^\infty(\mathbb{R}^d)\}$, where $B^{+0}_f\Phi=\Phi$, for all $f\in L^2(\mathbb{R}^d)\cap L^\infty(\mathbb{R}^d)$. From \cite{AcDh2} it follows that $\Gamma_2(L^2(\mathbb{R}^d)\cap L^\infty(\mathbb{R}^d))$ is an interacting Fock space. Moreover, the scalar product between two $n$-particle vectors is given by the following (cf \cite{AcDh1}).
\begin{proposition} \label{prop1}For all $f,\,g\in L^2(\mathbb{R}^d)\cap L^\infty(\mathbb{R}^d)$, one has
\begin{eqnarray*}
\langle B^{+n}_f\Phi,B^{+n}_g\Phi\rangle&=&c\sum^{n-1}_{k=0}2^{2k+1}{n!(n-1)!\over((n-k-1)!)^2}\,\langle f^{k+1}, g^{k+1}\rangle\\
\;\;\;\;\;&&\langle B^{+(n-k-1)}_f\Phi,B^{+(n-k-1)}_g\Phi\rangle.
\end{eqnarray*}
\end{proposition}

The quadratic exponential vector of an element $f\in L^2(\mathbb{R}^d)\cap L^\infty(\mathbb{R}^d)$, if it exists, is given by
$$
\Psi(f)=\sum_{n\geq0}\frac{B^{+n}_f\Phi}{n!}
$$ 
where by definition
\begin{equation}\label{Psi(0)Phi} 
\Psi(0)= B^{+0}_f\Phi = \Phi.
\end{equation}

In \cite{AcDh1}, it is proved that the quadratic exponential vector $\Psi(f)$ exists if $\|f\|_\infty<\frac{1}{2}$, and does not exists if  $\|f\|_\infty>\frac{1}{2}$. Furthermore, the scalar product between two exponential vectors, $\Psi(f)$ and $\Psi(g)$, is given by
\begin{equation}\label{Form}
\langle \Psi(f),\Psi(g)\rangle=e^{-\frac{c}{2}\int_{\mathbb{R}^d}\ln(1-4\bar{f}(s)g(s))ds}.
\end{equation}

Now, we refer to \cite{AcDh1} for the proof of the following theorem.
\begin{theorem}
The quadratic exponential vectors are linearly independents. Moreover, the set of quadratic exponential vectors is a total set in $\Gamma_2(L^2(\mathbb{R}^d)\cap L^\infty(\mathbb{R}^d))$.
\end{theorem}

\subsection{Quadratic second quantization}

For all linear operator $T$ on $L^2(\mathbb{R}^d)\cap L^\infty(\mathbb{R}^d)$, we define its quadratic second quantization, if it is well defined, by 
$$\Gamma_2(T)\Psi(f)=\Psi(Tf)$$
for all $f\in L^2(\mathbb{R}^d)\cap L^\infty(\mathbb{R}^d)$. 

Note that in \cite{AcDh2}, the authors have proved that  if $\Gamma_2(T)$ is well defined on the set of the quadratic exponential vectors, then $T$ is a contraction on $L^2(\mathbb{R}^d)\cap L^\infty(\mathbb{R}^d)$ equipped with the norm $\|.\|_\infty$. Conversely, if $T$ is a contraction on $L^2(\mathbb{R}^d)\cap L^\infty(\mathbb{R}^d)$ equipped with the norm $\|.\|_\infty$, then $\Gamma_2(T)$ is well defined on the set of the quadratic exponential vectors $\Psi(f)$ such that $\|f\|_\infty<\frac{1}{2}$. Moreover, they have characterized the operators $T$ on $L^2(\Bbb R^d)\cap L^\infty(\Bbb R^d)$ whose quadratic quantization is isometric (resp. unitary). The boundedness of $\Gamma_2(T)$ was also investigated.

\section{Main result}\label{section2}

Given a contraction $p$ on $L^2(\mathbb{R}^d)\cap L^\infty(\mathbb{R}^d)$ with respect to $\|.\|_\infty$, the aim of this section is to prove under which condition $\Gamma_2(p)$ is an orthogonal projection on $\Gamma_2(L^2(\mathbb{R}^d)\cap L^\infty(\mathbb{R}^d))$. The following lemma is an improvement of Lemma 1 of \cite{AcDh1}.
\begin{lemma}\label{lem}
For all $f,g\in L^2(\mathbb{R}^d)\cap L^\infty(\mathbb{R}^d)$, one has
$$\langle B^{+n}_f\Phi,B^{+n}_g\Phi\rangle=n!\frac{d^n}{dt^n}\Big|_{t=0}\langle\Psi(\sqrt{t}f),\Psi(\sqrt{t}g)\rangle.$$
\end{lemma}
\begin{proof}
Let $f,\,g\in L^2(\mathbb{R}^d)\cap L^\infty(\mathbb{R}^d)$ such that $\|f\|_\infty>0$ and $\|g\|_\infty>0$. Consider $0\leq t\leq \delta$, where 
$$\delta<\frac{1}{4}\inf\Big(\frac{1}{\|f\|_\infty^2},\frac{1}{\|g\|_\infty^2}\Big).$$ 
It is clear that $\|\sqrt{t}f\|_\infty< \frac{1}{2}$ and $\|\sqrt{t}g\|_\infty<\frac{1}{2}$. Moreover, one has
$$\langle\Psi(\sqrt{t}f),\Psi(\sqrt{t}g)\rangle=\sum_{m\geq0}\frac{t^m}{(m!)^2}\langle B^{+m}_f\Phi,B^{+m}_g\Phi\rangle.$$
Note that for all $m\geq n$, one has
\begin{eqnarray*}
&&\frac{d^n}{dt^n}\Big(\frac{t^m}{(m!)^2}\langle B^{+m}_{f}\Phi, B^{+m}_{g}\Phi\rangle\Big)\\
&&=\frac{m!t^{m-n}}{(m!)^2(m-n)!}\langle B^{+m}_{f}\Phi, B^{+m}_{g}\Phi\rangle\\
&&=\frac{t^{m-n}}{m!(m-n)!}\langle B^{+m}_{f}\Phi, B^{+m}_{g}\Phi\rangle.
\end{eqnarray*}
Put 
$$K_m=\frac{\delta^{m-n}}{m!(m-n)!}\|B^{+m}_{f}\Phi\| \|B^{+m}_{g}\Phi\|.$$
Then, from Proposition \ref{prop1}, it follows that
\begin{eqnarray*}
||B^{+m}_f\Phi||^2&=&c\sum_{k=0}^{m-1}2^{2k+1}\frac{m!(m-1)!}{((m-k-1)!)^2}|\|f^{k+1}\|^2_2\|B_f^{+(m-k-1)}\Phi\|^2\\
&=&c\sum_{k=1}^{m-1}2^{2k+1}\frac{m!(m-1)!}{((m-k-1)!)^2}|\|f^{k+1}\|^2_2\|B_f^{+(m-k-1)}\Phi\|^2\\
&&+2mc\|f\|^2_2\|B^{+(m-1)}_f\Phi\|^2\\
&=&c\sum_{k=0}^{m-2}2^{2k+3}\frac{m!(m-1)!}{(((m-1)-k-1)!)^2}|\|f^{k+2}\|^2_2\|B_f^{+((m-1)-k-1)}\Phi\|^2\\
&&+2mc\|f\|^2_2\|B^{+(m-1)}_f\Phi\|^2\\
&\leq&\Big(4m(m-1)\|f\|^2_\infty\Big)\Big[c\sum_{k=0}^{m-2}2^{2k+1}\frac{(m-1)!(m-2)!}{(((m-1)-k-1)!)^2}\|f^{k+1}\|^2_2\\
&&\;\;\;\;\|B_f^{+((m-1)-k-1)}\Phi\|^2\Big].
\end{eqnarray*}
But, one has
$$||B^{+(m-1)}_f\Phi||^2=c\sum_{k=0}^{m-2}2^{2k+1}\frac{(m-1)!(m-2)!}{((m-1)-k-1)!)^2}|\|f^{k+1}\|^2_2\|B_f^{+((m-1)-k-1)}\Phi\|^2.$$
Therefore, one gets 
$$||B^{+m}_f\Phi||^2\leq\Big[4m(m-1)\|f\|^2_\infty+2m\|f\|^2_2\Big]\|B^{+(m-1)}_f\Phi\|^2.$$
This proves that
$$\frac{K_m}{K_{m-1}} \leq \frac{\sqrt{4m(m-1)\|f\|^2_\infty+2m\|f\|^2_2}\sqrt{4m(m-1)\|g\|^2_\infty+2m\|g\|^2_2}}{m(m-n)}\,\delta.$$
It follows that
$$\lim_{m\rightarrow\infty}\frac{K_m}{K_{m-1}} \leq 4\|f\|_\infty\|g\|_\infty\delta<1.$$
Hence, the series $\sum_{m}K_m$ converges. Finally, we have proved that
\begin{eqnarray}\label{chi}
\frac{d^n}{dt^n}\langle\Psi(\sqrt{t}f),\Psi(\sqrt{t}g)\rangle=\sum_{m\geq n}\frac{t^{m-n}}{m!(m-n)!}\langle B^{+m}_{f}\Phi, B^{+m}_{g}\Phi\rangle.
\end{eqnarray}
Thus, by taking $t=0$ in the right hand side of (\ref{chi}), the result of the above lemma holds.
\end{proof}

As a consequence of the above lemma, we prove the following.
\begin{lemma}\label{lemm}
Let $p$ be a contraction on $L^2(\mathbb{R}^d)\cap L^\infty(\mathbb{R}^d)$ with respect to the norm $\|.\|_\infty$. If $\Gamma_2(p)$ is an orthogonal projection on $\Gamma_2(L^2(\mathbb{R}^d)\cap L^\infty(\mathbb{R}^d))$, then one has
$$\langle B^{+n}_{p(f)}\Phi, B^{+n}_{p(g)}\Phi\rangle=\langle B^{+n}_{p(f)}\Phi, B^{+n}_{g}\Phi\rangle=\langle B^{+n}_{f}\Phi, B^{+n}_{p(g)}\Phi\rangle$$
for all $f,\,g\in L^2(\mathbb{R}^d)\cap L^\infty(\mathbb{R}^d)$ and all $n\geq1$.
\end{lemma}
\begin{proof}
It $\Gamma_2(p)$ is an orthogonal projection on $\Gamma_2(L^2(\mathbb{R}^d)\cap L^\infty(\mathbb{R}^d))$, then 
\begin{eqnarray}\label{ab}
\langle\Psi(\sqrt{t}p(f)),\Psi(\sqrt{t}p(g))\rangle=\langle\Psi(\sqrt{t}p(f)),\Psi(\sqrt{t}g)\rangle=\langle\Psi(\sqrt{t}f),\Psi(\sqrt{t}p(g))\rangle
\end{eqnarray}
for all $f,\,g\in L^2(\mathbb{R}^d)\cap L^\infty(\mathbb{R}^d)$ (with $\|f\|_\infty>0$ and $\|g\|_\infty>0$) and all $0\leq t\leq\delta$ such that
$$\delta<\frac{1}{4}\inf\Big(\frac{1}{\|f\|_\infty^2},\frac{1}{\|g\|_\infty^2}\Big).$$ 
Therefore, the result of the above lemma follows from Lemma \ref{lem} and identity (\ref{ab}).
\end{proof}

Lemma \ref{lemm} ensures that the following result holds true.
\begin{lemma}\label{lemmm}
Let $p$ be a contraction on $L^2(\mathbb{R}^d)\cap L^\infty(\mathbb{R}^d)$ with respect to the norm $\|.\|_\infty$. If $\Gamma_2(p)$ is an orthogonal projection on $\Gamma_2(L^2(\mathbb{R}^d)\cap L^\infty(\mathbb{R}^d))$, then one has
\begin{eqnarray}\label{sassi}
\langle (p(f))^n,(p(g))^n\rangle=\langle f^n,(p(g))^n\rangle= \langle (p(f))^n,g^n\rangle
\end{eqnarray}
for all $f,\,g\in L^2(\mathbb{R}^d)\cap L^\infty(\mathbb{R}^d)$ and all $n\geq1$.
\end{lemma}
\begin{proof}
Suppose that $\Gamma_2(p)$ is an orthogonal projection on $\Gamma_2(L^2(\mathbb{R}^d)\cap L^\infty(\mathbb{R}^d))$. Then, in order to prove the above lemma we have to use induction.

- For $n=1$: Lemma \ref{lemm} implies that
$$\langle B^{+}_{p(f)}\Phi, B^{+}_{p(g)}\Phi\rangle=\langle B^{+}_{p(f)}\Phi, B^{+}_{g}\Phi\rangle=\langle B^{+}_{f}\Phi, B^{+}_{p(g)}\Phi\rangle$$
for all $f,\,g\in L^2(\mathbb{R}^d)\cap L^\infty(\mathbb{R}^d)$. Using the fact that $B_f\Phi=0$ for all $f\in L^2(\mathbb{R}^d)\cap L^\infty(\mathbb{R}^d)$ and the commutation relations in (\ref{commutation}) to get
\begin{eqnarray*}
\langle B^{+}_{p(f)}\Phi, B^{+}_{p(g)}\Phi\rangle&=&\langle \Phi,B_{p(f)}B^{+}_{p(g)}\Phi\rangle=2c\langle p(f),p(g)\rangle\\ 
\langle B^{+}_{p(f)}\Phi, B^{+}_{g}\Phi\rangle&=&\langle \Phi,B_{p(f)}B^{+}_{g}\Phi\rangle=2c\langle p(f),g\rangle\\
\langle B^{+}_{f}\Phi, B^{+}_{p(g)}\Phi\rangle&=&\langle \Phi,B_{f}B^{+}_{p(g)}\Phi\rangle=2c\langle f,p(g)\rangle.
\end{eqnarray*}
This proves that identity (\ref{sassi}) holds true for $n=1$.

- Let $n\geq1$ and suppose that identity (\ref{sassi}) is satisfied. Then, from Lemma \ref{lemm}, it follows that for all $f,\,g\in L^2(\mathbb{R}^d)\cap L^\infty(\mathbb{R}^d)$ one has
\begin{eqnarray}\label{rg}
\langle B^{+(n+1)}_{p(f)}\Phi, B^{+(n+1)}_{p(g)}\Phi\rangle=\langle B^{+(n+1)}_{p(f)}\Phi, B^{+(n+1)}_{g}\Phi\rangle=\langle B^{+(n+1)}_{f}\Phi, B^{+(n+1)}_{p(g)}\Phi\rangle.
\end{eqnarray}
Identity (\ref{rg}) and Proposition \ref{prop1} imply that
\begin{eqnarray}\label{roch}
\langle B^{+(n+1)}_{p(f)}\Phi, B^{+(n+1)}_{p(g)}\Phi\rangle&=&2^{2n+3}c n!(n+1)!\langle (p(f))^{n+1}, (p(g))^{n+1}\rangle\nonumber\\
&&+c\sum^{n-1}_{k=0}2^{2k+1}{n!(n+1)!\over((n-k)!)^2}\,\langle (p(f))^{k+1}, (p(g))^{k+1}\rangle\nonumber\\
&&\;\;\;\;\;\;\;\langle B^{+(n-k)}_{p(f)}\Phi,B^{+(n-k)}_{p(g)}\Phi\rangle\nonumber\\
&=&2^{2n+3}c n!(n+1)!\langle f^{n+1}, (p(g))^{n+1}\rangle\nonumber\\
&&+c\sum^{n-1}_{k=0}2^{2k+1}{n!(n+1)!\over((n-k)!)^2}\,\langle f^{k+1}, (p(g))^{k+1}\rangle\\
&&\;\;\;\;\;\;\;\langle B^{+(n-k)}_f\Phi,B^{+(n-k)}_{p(g)}\Phi\rangle\nonumber\\
&=&2^{2n+3}c n!(n+1)!\langle (p(f))^{n+1}, g^{n+1}\rangle\nonumber\\
&&+c\sum^{n-1}_{k=0}2^{2k+1}{n!(n+1)!\over((n-k)!)^2}\,\langle (p(f))^{k+1}, g^{k+1}\rangle\nonumber\\
&&\;\;\;\;\;\;\;\langle B^{+(n-k)}_{p(f)}\Phi,B^{+(n-k)}_g\Phi\rangle.\nonumber
\end{eqnarray}
Note that by induction assumption, one has
\begin{eqnarray}\label{sam}
\langle (p(f))^{k+1}, (p(g))^{k+1}\rangle=\langle f^{k+1}, (p(g))^{k+1}\rangle=\langle (p(f))^{k+1}, g^{k+1}\rangle
\end{eqnarray}
for all $k=0,\dots,n-1$. Therefore, from Lemma \ref{lemm} and identity (\ref{sam}), one gets 
\begin{eqnarray*}
&&c\sum^{n-1}_{k=0}2^{2k+1}{n!(n+1)!\over((n-k)!)^2}\,\langle (p(f))^{k+1}, (p(g))^{k+1}\rangle\langle B^{+(n-k)}_{p(f)}\Phi,B^{+(n-k)}_{p(g)}\Phi\rangle\\
&&=c\sum^{n-1}_{k=0}2^{2k+1}{n!(n+1)!\over((n-k)!)^2}\,\langle f^{k+1}, (p(g))^{k+1}\rangle\langle B^{+(n-k)}_f\Phi,B^{+(n-k)}_{p(g)}\Phi\rangle\\
&&=c\sum^{n-1}_{k=0}2^{2k+1}{n!(n+1)!\over((n-k)!)^2}\,\langle (p(f))^{k+1}, g^{k+1}\rangle\langle B^{+(n-k)}_{p(f)}\Phi,B^{+(n-k)}_g\Phi\rangle.
\end{eqnarray*}
Finally, from (\ref{roch}) one can conclude.
\end{proof}

Note that the set of contractions $p$ on $L^2(\mathbb{R}^d)\cap L^\infty(\mathbb{R}^d)$ with respect to $\|.\|_\infty$, such that $\Gamma_2(p)$ is an orthogonal projection on $\Gamma_2(L^2(\mathbb{R}^d)\cap L^\infty(\mathbb{R}^d))$, is reduced to the following.  
\begin{lemma}\label{bar}
Let $p$ be a contraction on $L^2(\mathbb{R}^d)\cap L^\infty(\mathbb{R}^d)$ with respect to the norm $\|.\|_\infty$. If $\Gamma_2(p)$ is an orthogonal projection on $\Gamma_2(L^2(\mathbb{R}^d)\cap L^\infty(\mathbb{R}^d))$, then
$$p(\bar{f})=\overline{p(f)}.$$
for all $f\in L^2(\mathbb{R}^d)\cap L^\infty(\mathbb{R}^d)$. 
 \end{lemma}
\begin{proof}
Let $p$ be a contraction on $L^2(\mathbb{R}^d)\cap L^\infty(\mathbb{R}^d)$ with respect to the norm $\|.\|_\infty$ such that $\Gamma_2(p)$ is an orthogonal projection on $\Gamma_2(L^2(\mathbb{R}^d)\cap L^\infty(\mathbb{R}^d))$. Then, from Lemma \ref{lemmm} it is clear that $p=p^*=p^2$ (taking $n=1$ in (\ref{sassi})). Moreover, for all $f_1,f_2,g_1,g_2\in L^2(\mathbb{R}^d)\cap L^\infty(\mathbb{R}^d)$, one has
\begin{eqnarray*}
\langle (p(f_1+f_2))^2,\,(g_1+g_2)^2\rangle=\langle (p(f_1+f_2))^2,\,(p(g_1+g_2))^2\rangle.
\end{eqnarray*} 
It follows that
\begin{eqnarray}\label{p}
&&\langle (p(f_1))^2,g_1^2+g_2^2\rangle+\langle (p(f_2))^2,g_1^2+g_2^2\rangle+4\langle p(f_1)p(f_2),g_1g_2\rangle\nonumber\\
&&=\langle (p(f_1))^2,(p(g_1))^2+(p(g_2))^2\rangle+\langle (p(f_2))^2,(p(g_1))^2+(p(g_2))^2\rangle\\
&&+4\langle p(f_1)p(f_2),p(g_1)p(g_2)\rangle.\nonumber
\end{eqnarray}
Then, using (\ref{sassi}) and (\ref{p}) to obtain 
\begin{eqnarray}\label{bah}
\langle p(f_1)p(f_2),g_1g_2\rangle=\langle p(f_1)p(f_2),p(g_1)p(g_2)\rangle
\end{eqnarray}
for all $f_1,f_2,g_1,g_2\in L^2(\mathbb{R}^d)\cap L^\infty(\mathbb{R}^d)$. Now, denote by $\mathcal{M}_a$ the multiplication operator by the function $a\in L^2(\mathbb{R}^d)\cap L^\infty(\mathbb{R}^d)$. Then, identity (\ref{bah}) implies that
$$\langle \mathcal{M}_{p(f_2)\bar{g_2}}p(f_1),g_1\rangle=\langle \mathcal{M}_{p(f_2)\overline{p(g_2)}}p(f_1),p(g_1)\rangle$$
for all $f_1,f_2,g_1,g_2\in L^2(\mathbb{R}^d)\cap L^\infty(\mathbb{R}^d)$. This gives that
\begin{equation}\label{multiplication}
\mathcal{M}_{p(f_2)\bar{g_2}}p=p\mathcal{M}_{p(f_2)\overline{p(g_2)}}p
\end{equation}
for all $f_2, g_2\in L^2(\mathbb{R}^d)\cap L^\infty(\mathbb{R}^d)$. Taking the adjoint in (\ref{multiplication}), one gets
\begin{equation}\label{adjoint}
p\mathcal{M}_{\overline{p(f_2)}\,g_2}=p\mathcal{M}_{\overline{p(f_2)}p(g_2)}p.
\end{equation}
Note that, for all $f,g\in L^2(\mathbb{R}^d)\cap L^\infty(\mathbb{R}^d)$, identity (\ref{multiplication}) implies that
\begin{eqnarray}\label{identification}
\mathcal{M}_{p(f)\bar{g}}p=p\mathcal{M}_{p(f)\,\overline{p(g)}}p.
\end{eqnarray}
Moreover, from (\ref{adjoint}), one has
\begin{eqnarray}\label{iden}
p\mathcal{M}_{p(f)\,\overline{p(g)}}p=p\mathcal{M}_{f\,\overline{p(g)}}.
\end{eqnarray}
Therefore, identities (\ref{identification}) and (\ref{iden}) yield
\begin{equation}\label{dual}
\mathcal{M}_{p(f)\bar{g}}p=p\mathcal{M}_{f\,\overline{p(g)}}
\end{equation}
for all $f,g\in L^2(\mathbb{R}^d)\cap L^\infty(\mathbb{R}^d)$. Hence, for all $f,g,h\in L^2(\mathbb{R}^d)\cap L^\infty(\mathbb{R}^d)$, identity (\ref{dual}) gives
\begin{equation}\label{dhahri}
p(f\,\overline{p(g)}h)=p(f)\bar{g}p(h).
\end{equation}
Taking $f=g=p(h)$ in (\ref{dual}) to get
\begin{equation}\label{17}
\mathcal{M}_{|p(h)|^2}p=p\mathcal{M}_{|p(h)|^2}.
\end{equation}
Then, if we put $f=h=p(g)$ in (\ref{dhahri}), one has
$$p(\overline{p(g)}\,p(g)^2)=p(|p(g)|^2p(g))=(p(g))^2\bar{g}.$$
But, from (\ref{17}), one has
$$p(|p(g)|^2p(g))=(p\mathcal{M}_{|p(g)|^2})(p(g))=\mathcal{M}_{|p(g)|^2}p(p(g))=|p(g)|^2p(g).$$
Hence, one obtains
\begin{equation}\label{fin}
|p(g)|^2p(g)=(p(g))^2\bar{g}
\end{equation}
for all $g\in L^2(\mathbb{R}^d)\cap L^\infty(\mathbb{R}^d)$. Now, let $g$ be a real function in $L^2(\mathbb{R}^d)\cap L^\infty(\mathbb{R}^d)$. So, the polar decomposition of $p(g)$ is given by
$$p(g)=|p(g)|e^{i\theta_{p(g)}}.$$
Thus, identity (\ref{fin}) implies that 
$$|p(g)|^3e^{-i\theta_{p(g)}}=|p(g)|^2g.$$
This proves that for all $x\in\mathbb{R}^d$, $\theta_{p(g)}(x)=k_x\pi$, $k_x\in\mathbb{Z}$. Therefore, $p(g)$ is a real function. Now, taking $f=f_1+if_2\in L^2(\mathbb{R}^d)\cap L^\infty(\mathbb{R}^d)$, where $f_1,f_2$ are real functions on $\mathbb{R}^d$. It is clear that 
$$p(\bar{f})=p(f_1-if_2)=\overline{p(f_1)+ip(f_2)}=\overline{p(f)}.$$
This completes the proof of the above lemma.
\end{proof}

As a consequence of Lemmas \ref{lemmm} and \ref{bar} we prove the following theorem.
\begin{theorem}
Let $p$ be a contraction on $L^2(\mathbb{R}^d)\cap L^\infty(\mathbb{R}^d)$ with respect to the norm $\|.\|_\infty$. Then, $\Gamma_2(p)$ is an orthogonal projection on $\Gamma_2(L^2(\mathbb{R}^d)\cap L^\infty(\mathbb{R}^d))$ if and only if $p=\mathcal{M}_{\chi_I}$, where $\mathcal{M}_{\chi_I}$ is a multiplication operator by a characteristic function $\chi_I$, $I\subset\mathbb{R}^d$.
\end{theorem}
\begin{proof}
Note that if $p=\mathcal{M}_{\chi_I}$, $I\subset\mathbb{R}^d$, then from identity (\ref{Form}) it is clear that
\begin{eqnarray*}
e^{-\frac{c}{2}\int_{I}\ln(1-4\bar{f}(s)g(s))ds}&=&\langle \Psi_2(p(f)),\Psi_2(g)\rangle\\
&=&\langle \Psi_2(f),\Psi_2(p(g))\rangle\\
&=&\langle \Psi_2(p(f)),\Psi_2(p(g))\rangle
\end{eqnarray*}
for all $f,\,g\in L^2(\mathbb{R}^d)\cap L^\infty(\mathbb{R}^d)$ such that $\|f\|_\infty<\frac{1}{2}$ and $\|g\|_\infty<\frac{1}{2}$. Hence, $\Gamma_2(p)$ is an orthogonal projection on $\Gamma_2(L^2(\mathbb{R}^d)\cap L^\infty(\mathbb{R}^d))$.

Now, suppose that $\Gamma_2(p)$ is an orthogonal projection on $\Gamma_2(L^2(\mathbb{R}^d)\cap L^\infty(\mathbb{R}^d))$. Then, Lemma \ref{lemmm} implies that
\begin{eqnarray*}
\langle (p(f))^n,(p(g))^n\rangle=\langle f^n,(p(g))^n\rangle= \langle (p(f))^n,g^n\rangle
\end{eqnarray*}
for all $f,\,g\in L^2(\mathbb{R}^d)\cap L^\infty(\mathbb{R}^d)$ and all $n\geq1$. In particular, if $n=2$ one has 
\begin{eqnarray*}
\langle (p(f_1+\bar{f_2}))^2, g^2\rangle=\langle (f_1+\bar{f_2})^2,(p(g))^2\rangle
\end{eqnarray*} 
for all $f_1,\,f_2,\,g\in L^2(\mathbb{R}^d)\cap L^\infty(\mathbb{R}^d)$. This gives
\begin{eqnarray}\label{samiha}
&&\langle (p(f_1))^2, g^2\rangle+2\langle p(f_1)p(\bar{f_2}), g^2\rangle+
\langle (p(\bar{f_2}))^2, g^2\rangle\nonumber\\
&&=\langle f_1^2, (p(g))^2\rangle+2\langle f_1\bar{f_2}, (p(g))^2\rangle+
\langle (\bar{f_2})^2, (p(g))^2\rangle.
\end{eqnarray}
Using identity (\ref{samiha}) and Lemma \ref{lemmm} to get 
$$\langle p(f_1)p(\bar{f_2}), g^2\rangle=\langle f_1\bar{f_2}, (p(g))^2\rangle.$$
This yields 
\begin{eqnarray}\label{erri}
\int_{\mathbb{R}^d}\bar{f_1}(x)f_2(x)(p(g))^2(x)dx=\int_{\mathbb{R}^d}\overline{p(f_1)}(x)p(f_2)(x)g^2(x)dx.
\end{eqnarray}
But, from Lemma \ref{bar}, one has $\overline{p(f)}=p(\bar{f})$, for all $f\in L^2(\mathbb{R}^d)\cap L^\infty(\mathbb{R}^d)$. Then, identity (\ref{erri}) implies that
\begin{eqnarray*}
\langle f_1, M_{(p(g))^2}f_2\rangle=\langle f_1, (pM_{g^2}p)f_2\rangle
\end{eqnarray*}
for all $f,\,g\in L^2(\mathbb{R}^d)\cap L^\infty(\mathbb{R}^d)$. Hence, one obtains
\begin{eqnarray}\label{war}
M_{(p(g))^2}=pM_{g^2}p.
\end{eqnarray}
In particular, for $g=\chi_I$ where $I\subset\mathbb{R}^d$, one has
\begin{equation}\label{hed}
\mathcal{M}_{(p(\chi_I))^2}=p\mathcal{M}_{\chi_I}p.
\end{equation}
If $I$ tends to $\mathbb{R}^d$, the operator $\mathcal{M}_{\chi_I}$ converges to $id$ (identity of $L^2(\mathbb{R}^d)\cap L^\infty(\mathbb{R}^d)$) for the strong topology. From (\ref{hed}), it follows that
$$p(f)=p^2(f)=\lim_{I\uparrow\mathbb{R}^d}\mathcal{M}_{(p(\chi_I))^2}f$$
for all $f\in L^2(\mathbb{R}^d)\cap L^\infty(\mathbb{R}^d)$. But, the set of multiplication operators is a closed set for the strong topology. This proves that $p=\mathcal{M}_{a}$, where $a\in L^\infty(\mathbb{R}^d)$. Note that $p=p^2$ is a positive operator. This implies that $a$ is a positive function. Moreover, one has $p^n=p$ for all $n\in\mathbb{N}^*$. This gives $\mathcal{M}_{a^n}=\mathcal{M}_{a}$ for all $n\in\mathbb{N}^*$. It follows that $a^n=a$ for all $n\in\mathbb{N}^*$. Therefore, the operator $a$ is necessarily a characteristic function on $\mathbb{R}^d$.  
\end{proof}

\bigskip
{\bf\large Acknowledgments}\bigskip

I gratefully acknowledge stimulating discussions with Eric Ricard. I would like also to thank Rolando Rebolledo for his hospitality, for reading the paper and for interesting comments.

\end{document}